\documentclass[12pt]{amsart}

\usepackage{amsmath,amscd,amsthm,amssymb,amsxtra,latexsym,epsfig,epic,graphics,mathdots,amssymb, enumerate}
\usepackage[all]{xy}
\SelectTips{cm}{11}
\usepackage[utf8]{inputenc}
\usepackage[svgnames]{xcolor}
\usepackage{tcolorbox}
\tcbuselibrary{skins,listings,theorems}
\usepackage{hyperref}

\usepackage{cite}
\makeatletter
\newcommand{\citecomment}[2][]{\citenum{#2}#1\citevar}
\newcommand{\citeone}[1]{\citecomment{#1}}
\newcommand{\citetwo}[2][]{\citecomment[,~#1]{#2}}
\newcommand{\citevar}{\@ifnextchar\bgroup{;~\citeone}{\@ifnextchar[{;~\citetwo}{]}}}
\newcommand{\citefirst}{\@ifnextchar\bgroup{\citeone}{\@ifnextchar[{\citetwo}{]}}}
\newcommand{\cites}{[\citefirst}
\makeatother

\usepackage[all]{xy}

\newtheorem{thm}{Theorem}

\newtheorem{prop}[thm]{Proposition}

\theoremstyle{definition}
\newtheorem{defn}[thm]{Definition}

\newtheorem{exmp}[thm]{Example}

\newtheorem{ques}[thm]{Question}    

\newtheorem{rem}[thm]{Remark}          

\newtheorem*{ack}{Acknowledgments}      

\newtheorem{defn-thm}[thm]{Definition--Theorem}  
\newtheorem{defn-lem}[thm]{Definition--Lemma}  

\theoremstyle{remark}

\renewcommand{\c}[0]{{\mathbb C}}

\newcommand{\z}[0]{{\mathbb Z}}

\newcommand{\p}[0]{{\mathbb P}}

\newcommand{\pic}[0]{\operatorname{Pic}}

\newcommand{\Gr}[0]{\mathrm{Gr}}

\newcommand{\ext}[0]{\operatorname{Ext}}    
\newcommand{\Hom}[0]{\operatorname{Hom}}

\def\loccoh#1.#2.#3.#4.{H^{#1}_{#2}(#3,#4)}

\DeclareMathAlphabet{\mathchanc}{OT1}{pzc}%
                                {m}{it}

\begin{document}
\bibliographystyle{amsalpha}



\title[Resonance, syzygies and Ulrich bundles on $V_5$]{Resonance, syzygies, and rank--$3$ Ulrich bundles on the del Pezzo threefold $V_5$}
\author{Marian Aprodu}
\address{University of Bucharest, Department of Mathematics \& Institute of Mathematics ``Simion Stoilow'' of the Romanian Academy}
\email{marian.aprodu@fmi.unibuc.ro \\ marian.aprodu@imar.ro}

\author{Yeongrak Kim}
\address{Department of Mathematics \& Institute of Mathematical Science, Pusan National University, 2 Busandaehak-ro 63beon-gil, Geumjeong-gu, 46241 Busan, Korea}
\email{yeongrak.kim@pusan.ac.kr}

\begin{abstract}
We investigate a geometric criterion for a smooth curve $C$ of genus $14$ and degree $18$ to be described as the zero locus of sections in an Ulrich bundle of rank $3$ on a del Pezzo threefold $V_5 \subset \p^6$. The main challenge is to read off the Pfaffian quadrics defining $V_5$ from geometric structures of $C$. We find that this problem is related to the existence of a special rank-two vector bundle on $C$ with trivial resonance.
It gives a description of the image of the rational map that appeared in a work of Ciliberto-Flamini-Knutsen, for the case of degree $5$ del Pezzo threefolds.
From an explicit calculation of the Betti table of such a curve, we also deduce the uniqueness of the del Pezzo threefold containing a given curve.
\end{abstract}

\dedicatory{Dedicated to the memory of Gianfranco Casnati}


\maketitle

\section{Introduction}
Let $X \subset \p^N$ be a projective variety of dimension $n$. A coherent sheaf $\mathcal E$ on $X$ is called an \emph{Ulrich sheaf} if there is a finite linear projection $\pi : \p^N \dashrightarrow \p^n$ which is defined over $X$ such that $\pi_{\ast} \mathcal E \simeq \mathcal O_{\p^n}^{\oplus k}$ for some positive integer $k$ \cite{ES03}. In commutative algebra, an analogous object often called an \emph{Ulrich module} has appeared in the form of maximal Cohen-Macaulay modules having the maximal number of generators, and has been prominently featured in modern mathematics due to its various spectrum of applications, see for instance \cite{Ulr84, BGS87, IMW22, Ma23}. Eisenbud and Schreyer transferred this concept in the geometric setup and discovered a remarkable connection with Cayley--Chow forms \cite{ES03}; see also \cite{Bea18, CMP21} for recent developments from the geometric perspective and further applications. 

The most significant open problem in the Ulrich sheaf theory is the \emph{Eisenbud--Schreyer conjecture} which predicts that any smooth projective variety must carry an Ulrich bundle. If true, this conjecture would have dramatic consequences on the expected shapes of cohomology tables \cite{ES03,ES11}. The next step, assuming the existence was established, is to determine the possible ranks of Ulrich bundles and classify them. 
While all these problems are widely open, they have been resolved in a number of relevant cases, in particular, for varieties of minimal degree \cite{AHMP19, ES03}.
Hence, it is natural to raise the same questions in the next case, namely, for del Pezzo varieties. These problems are rather straightforward for smooth del Pezzo varieties of dimension $n=1, 2$ \cites[Corollary 4.5, 6.5]{ES03}[Proposition 5]{Bea18}. More recently, the answer was also given for del Pezzo threefolds of degree $d$ as in various case studies: \cite{Bea00, CH12, LMS15} for $d=3$; \cite{CKL21} for $d=4$; \cite{Fae05, LP21} for $d=5$; \cite{CFM16} and \cite{CFM18} for two cases of $d=6$; \cite{CFM17} for $d=7$; and \cite[Corollary 5.3, Proposition 5.11]{ES03} for $d=8$. In particular, $X$ carries Ulrich bundles of every rank $r \ge 2$ unless $X = \nu_2 (\p^3) \subset \p^9$ is the $2$-uple embedding of $\p^3$ where it only carries Ulrich bundles of even ranks. Very recently, Ciliberto, Flamini, and Knutsen found a single deformation argument for constructing Ulrich bundles of rank $2$ and $3$ over every del Pezzo threefold of degree $3 \le d \le 7$ \cite{CFK23}. 

There were several attempts to construct Ulrich bundles directly using the Hartshorne-Serre correspondence, for example, \cite{AC00} over a del Pezzo threefold of Picard number $1$ and \cite[Proposition 8]{Bea18} over every del Pezzo threefold for rank $2$ Ulrich bundles. The problem translates into the existence of an elliptic normal curve of degree $d+2$ contained in a del Pezzo threefold $V_d \subset \p^{d+1}$ of degree $d$. For rank--$3$ Ulrich bundles, we refer to \cite{CH12} for $d=3$, and to an unpublished additional note to \cite{CKL21} for $d=4$. In this case, we need to look for an arithmetically Cohen-Macaulay curve $C \subset V_d \subset \p^{d+1}$ of genus $2d+4$ and degree $3d+3$ such that the twisted canonical line bundle $\omega_C (-1)$ has two global sections which generate the module of twisted global sections $H^0_{\ast} (\omega_C) = \bigoplus_{j \in \z} H^0 (C, \omega_C(j))$. In loc. cit., the authors showed that a general curve $C$ of genus $2d+4$ with $d=3$ or $d=4$ admits an embedding $C \hookrightarrow V_d \subset \p^{d+1}$ satisfying these conditions, and the incidence set $\{ (C, V_d) ~ \vert ~ C \subset V_d \subset \p^{d+1}\}$ dominates the family of del Pezzo threefolds $\{V_d\}$ of degree $d$. In particular, a general del Pezzo threefold of degree $d$ has an Ulrich bundle of rank $3$. We remark that the results of \cite{CH12, CKL21} are deeply related to the rationality problem for the Hurwitz spaces $\mathcal H (k, 2g+2k-2)$ and the moduli of smooth curves $\mathcal M_{g}$ as in computational approaches developed by Gei{\ss} and Schreyer \cites{Gei12}[Appendix]{CH12}. 

The main purpose of this paper is to answer the analogous question for Ulrich bundles of rank $3$ on a del Pezzo threefold $V_5 \subset \p^6$ of degree $5$. 
As pointed out in \cite[Section 5]{CFK23}, a general curve of genus $14$ in $\p^6$ cannot be associated to an Ulrich bundle of rank $3$ on $V_5$, and this fact represents a significant difference from the cases $d=3$ and $d=4$. On the other hand, any model of $V_5 \subset \p^6$ has an Ulrich bundle $\mathcal E$ of rank $3$, and two generic global sections $s_1, s_2$ will give rise an arithmetically Cohen-Macaulay (ACM for short) curve $C \subset V_5 \subset \p^6$ of genus $14$, degree $18$ as the zero locus of $s_1 \wedge s_2$. Hence, it is natural to inquire about the geometry of the sublocus of $\mathcal M_{14}$ which parametrizes curves that induce Ulrich bundles of rank $3$ over some $V_5 \subset \p^6$. Following the notation in \cite[Section 5]{CFK23}, this question is the same as studying the image of $\varphi_{3,5} : \mathbb{G} \dashrightarrow \mathcal M_{14}$.

The outline of the paper is as follows. In Section \ref{sec:Rk3Ulrich} we recall the connection between rank--3 Ulrich bundles on $V_d$ and certain curves contained in $V_d$. For $d=5$, which is the focus of the present work, we compute the Betti table of these curves as in Proposition \ref{prop:BettiOfACMCurve}. Indeed, we provide an equivalent condition in terms of Koszul cohomology groups of $C \subset V_5 \subset \p^6$ for $C$ to arise from a rank--$3$ Ulrich bundle on (a model of) $V_5$.
Section \ref{sec:Res} is devoted to the main result of the paper, Theorem \ref{thm:MainThmV5}. It shows that the source of an embedding in $V_5$ of a curve $C$ of genus $14$ and degree $18$ in $\mathbb P^6$ is the existence of a rank--$2$ vector bundle on $C$ with $5$-dimensional space of global sections, determinant equal to $\mathcal{O}_C(1)$ and trivial resonance. In fact, the theorem proves that a rank--$2$ bundle with determinant $\mathcal{O}_C(1)$ and exactly five independent global sections gives an embedding of the curve in a linear section of the Grassmannian, and this section is smooth if and only if the resonance of the given bundle is trivial. For the proof of Theorem \ref{thm:MainThmV5} we need an algebraic characterization of the triviality of the resonance in terms of generalized Pfaffians, which is given in Theorem \ref{thm:ResonancePfaffianLemma}. We conclude the paper with some remarks on the main theorem.

\medskip

Throughout this paper, we work over $\c$ for simplicity.

\begin{ack}
 M. A. was partly supported by the PNRR	grant CF 44/14.11.2022 \emph{Cohomological Hall algebras of smooth surfaces and applications.} Y. K. is supported by Basic Science Research Program of the NRF of Korea (NRF-2022R1C1C1010052). The authors have greatly benefited from discussions with Gavril Farkas, to whom they express their special thanks. Y. K. thanks Mateusz Michalek for his hospitality and helpful discussion. 
\end{ack}

\section{Rank $3$ Ulrich bundles on del Pezzo threefolds of Picard number $1$}
\label{sec:Rk3Ulrich}

In this section, we recall the construction of rank--$3$ Ulrich bundles on del Pezzo threefolds of Picard number $1$. 

We start with the general definition of Ulrich bundles on a (smooth) projective variety $X \subset \mathbb{P}^N$. Note that it depends on an embedding i.e. on the choice of a polarization $(X, \mathcal O_X(1))$ where $\mathcal O_X(1) = \mathcal O_{\p^N}(1) \vert_X$ is a very ample line bundle on $X$.

\begin{defn}
Let $X \subseteq \mathbb{P}^N$ be a projective variety of dimension $n$, and let $\mathcal E$ be a nonzero coherent sheaf on $X$. 
\begin{enumerate}[(i)]
\item $\mathcal E$ is called \emph{ACM} if $\mathcal E$ is locally Cohen-Macaulay, and has no intermediate cohomology, that is, 
\[
H^i (X, \mathcal E(-j))=0
\]
for every $0<i<n$ and $j \in \z$.
\item $\mathcal E$ is called \emph{Ulrich} if $\mathcal E$ has the same cohomological behavior as the structure sheaf of $\p^n$, that is, 
\[
H^i (X, \mathcal E(-j)) = 0
\]
for every $i \in \z$ and $1 \le j \le n$.
\end{enumerate}
\end{defn}

This definition admits several variants, see for instance \cite[Proposition 2.1]{ES03}. We refer to \cite[Section 2]{CH12} for basic properties of Ulrich sheaves. Note that any Ulrich sheaf over a smooth projective variety is locally free, and hence in the smooth case it makes sense to speak about Ulrich bundles.

An $n$-dimensional projective variety $X \subseteq \mathbb{P}^N$ is called a \emph{del Pezzo variety} if $\omega_X \simeq \mathcal O_X (1 - \dim X)$. It is known that there are exactly four types of smooth del Pezzo threefolds of Picard number $1$, namely,
\begin{itemize}
\item  $V_3 \subset \p^4$ a cubic threefold;
\item  $V_4 \subset \p^5$ a complete intersection of two quadrics;
\item  $V_5 \subset \p^6$ a $3$-dimensional linear section of the Grassmannian $\Gr(2,5) \subset \p^9$;
\item  $v_2 (\p^3) \subset \p^9$ a $2$-uple embedding of $\p^3$.
\end{itemize}

The Ulrich geography for the Veronese threefold in $\mathbb P^9$ is very simple, it has a unique Ulrich bundle of rank $2$ and no Ulrich bundle of odd ranks \cite[Proposition 5.11, Corollary 5.3]{ES03}.

The remaining cases $V_d \subseteq \mathbb{P}^{d+1}$ are del Pezzo threefolds of degree $d=3,4,5$ embedded by a very ample line bundle $\mathcal O_{V_d}(1)$ which is a generator of the Picard group $\pic (V_d)$ in each of the cases. Arrondo and Costa classified possible Chern classes of ACM bundles of small ranks over $V_d$ \cite[Theorem 3.4, 4.9]{AC00}, and constructed some of them in their list. For instance, a general global section of an Ulrich bundle $\mathcal E$ of rank $2$ over $V_d$ degenerates along an elliptic normal curve in $\mathbb{P}^{d+1}$, which is of degree $d+2$. 

By the Hartshorne-Serre correspondence (see \cite{Arr07} for more details), the existence of rank--$2$ Ulrich bundles over $V_d$ reduces to the question whether or not there exist elliptic normal curves contained in $V_d$. We refer to \cites[Example 3.3]{AC00}[Proposition 8]{Bea18}~ for the construction of rank--$2$ Ulrich bundles on $V_d$ in this manner. 

Similarly, the existence problem for rank--$3$ Ulrich bundles over $V_d$ is translated in terms of the existence of certain curves lying on $V_d$. Let $\mathcal E$ be an Ulrich bundle of rank $3$ on a del Pezzo threefold $V_d \subset \p^{d+1}$ of degree $d$, and let $V \subset H^0 (V_d, \mathcal E)$ be a general subspace of dimension $2$. The inclusion $V \otimes \mathcal O_{V_d} \to \mathcal E$ degenerates along a smooth ACM curve $C_d$ of genus $2d+4$ and degree $3d+3$ (cf. \cite[Theorem 3.5]{CFK23}) such that $\omega_{C_d} (-1)$ is generated by a $2$-dimensional space $W$ and the multiplication map $\nu : W \otimes H^0 (\p^{d+1}, \mathcal O_{\p^{d+1}} (j)) \to H^0 (\omega_{C_d} (j-1))$ has full rank for every $j \in \z$. Thanks to the base-point-free pencil trick, the last condition is satisfied when the multiplication map $W \otimes H^0 (C_d, \mathcal O_{C_d}(1)) \to H^0 (C_d, \omega_{C_d})$ is surjective.

Conversely, if there is such a curve $C_d$ embedded in $V_d$, the Hartshorne-Serre correspondence implies the existence of a rank $3$ vector bundle $\mathcal E$ which fits into the short exact sequence
\[
0 \to \mathcal O_{V_d}^{\oplus 2} \to \mathcal E \to \mathcal I_{C_d / V_d} (3) \to 0.
\]
A direct computation shows that $\mathcal E$ is an Ulrich bundle on $V_d$. We remark that the generating condition on $\omega_{C_d} (-1)$ is essential, see \cite[Remark 5.5]{CH12} for a discussion on this issue. 

Constructions of rank--$3$ Ulrich bundles on $V_3$ and $V_4$ using syzygies of algebraic curves and the Hartshorne-Serre correspondence are well understood and are contained in the paper by Casanellas-Hartshorne-Gei{\ss}-Schreyer \cite[Proposition 5.4, Appendix A]{CH12}, and in an unpublished appendix to the paper by Cho, the second named author, and Lee \cite{CKL21}. The key idea behind these constructions was first made clear by Gei{\ss} in his Ph. D. dissertation \cite[Section 3.4.2]{Gei13}, in relation with the rationality problem for the moduli space of curves $\mathcal M_g$ and for the Hurwitz space $\mathcal H(k, 2g+2k-2)$, for small values of $g,k$.

\medskip

We briefly describe these constructions below. 
\begin{enumerate}[(i)]
\item Construct a random curve $C_d$ of genus $g=2d+4$ together with a base-point-free pencil $\mathfrak g^1_{d+3}$ using a bigraded projective model in $\p^1 \times \p^2$. Note that such a pencil of degree $d+3$ always exists thanks to the Brill-Noether theory.

\item The linear system $|\omega_{C_d} - \mathfrak g^1_{d+3}|$ is a very ample $\mathfrak g^{d+1}_{3d+3}$ on $C_d$. Verify that $C_d \subset \p^{d+1}$ is contained in a smooth del Pezzo threefold $V_d$.

\item Verify that $C_d$ induces an Ulrich bundle of rank $3$ via the Hartshorne-Serre correspondence.

\item Verify that the incidence family $\{ (C_d, V_d ) ~ \vert ~ C_d \subset V_d \}$ from this construction dominates the space of del Pezzo threefolds $\{ V_d \}$.
\end{enumerate}

As a result, we see that a general $V_d$ carries rank--$3$ Ulrich bundles for $d=3$ and for $d=4$.

\medskip

The case of the del Pezzo threefold $V_5 \subset \p^6$ of degree $5$, the central object in this paper, is treated separately. First note that $V_5 = \Gr(2, 5) \cap H_1 \cap H_2 \cap H_3 \subset \p^6$ is a smooth $\p^6$-section of the Grassmannian $\Gr (2,5) \subset \p^9$ of lines in $\p^4$. From the Euler sequence over the Grassmannian $\Gr(2,5)$, we have the following short exact sequence of vector bundles on $V_5$
\[
0 \to \mathcal U \to V \otimes \mathcal O_{V_5} \to \mathcal Q \to 0
\]
so that $\mathcal U$ is the restriction of the universal subbundle on $\Gr(2,5)$ and $\mathcal Q$ is the restriction of the quotient bundle on $Gr(2,5)$. Thanks to the Bott formula, one readily verifies that the bundles $\mathcal U$ and $\mathcal Q^{\vee}$ are ACM bundles on $V_5$. Furthermore, Orlov showed that there is a full exceptional collection in the derived category of coherent sheaves on $V_5$ \cite{Orl91}, namely,
\[
D^b (V_5) = \langle \mathcal U, \mathcal Q^{\vee}, \mathcal O_{V_5}, \mathcal O_{V_5} (1) \rangle.
\]
We refer to \cites[Section 3]{Fae05}[Section 2.3]{LP21}~  for more details on the notion and computations for cohomology groups of these bundles. 

It is well-known that $V_5$ supports an Ulrich bundle $\mathcal E$ of rank $3$, namely, $\mathrm{Sym}^2 (\mathcal U^{\vee})$ and its deformations. Note also that there is a $10$-dimensional family of stable Ulrich bundles of rank $3$ from the computation $h^1 (\mathcal E \otimes \mathcal E^{\vee}) = 10$ \cites[Example 4.4]{AC00}[Proposition 6.8]{Fae05}[Theorem 1.1]{LP21}. In particular, Ulrich bundles on $V_5$ have a very nice presentation using universal bundles as follows. 

\begin{prop} \cite[Proposition 3.4]{LP21} \label{prop:PresentationUlrichOnV5}
For any $r \ge 2$, an Ulrich bundle $\mathcal E$ of rank $r$ on $V_5$ corresponds to a quiver representation and fits into the short exact sequence
\[
0 \to \mathcal U^{\oplus r} \to {\mathcal Q^{\vee}}^{\oplus r} \to \mathcal E (-1) \to 0.
\]
\end{prop}

We now consider an analogous construction using the Hartshorne-Serre correspondence. Note that two generic global sections of an Ulrich bundle $\mathcal E$ of rank $3$ degenerates along an ACM curve $C \subset V_5$ which fits into the short exact sequence
\[
0 \to \mathcal O_{V_5}^{\oplus 2} \to \mathcal E \to \mathcal I_{C/V_5} (3) \to 0.
\]
From this short exact sequence, we may read off algebraic information about $C$ as follows.

\begin{prop}
\label{prop:BettiOfACMCurve} 
Notation as above.
$C$ is an ACM curve of genus $2d+4 = 14$ and degree $3d+3 = 18$ in $\p^6$. The Betti table of $C$ is given by 
\[
\begin{array}{c|cccccc}
j \setminus i & 0&1&2&3&4&5 \\
\hline
0 & 1&-&-&-&-&- \\
1 & -&5&5&-&-&- \\
2 & -&13&45&56&25&- \\
3 & -&-&-&-&-&2
\end{array}
\]
In particular, $C$ can only lie on one single del Pezzo threefold in $\mathbb P^6$.
\end{prop}

\begin{proof}
The minimal free resolutions of $2 \mathcal O_{V_5}$ and of $\mathcal E$ are
\[
0 \to 2 \mathcal O_{\p^6}(-5) \to 10 \mathcal O_{\p^6} (-3) \to 10 \mathcal O_{\p^6}(-2) \to 2 \mathcal O_{\p^6} \to 2 \mathcal O_{V_5} \to 0,
\]
and
\[
0 \to 15 \mathcal O_{\p^6}(-3) \to 45 \mathcal O_{\p^6}(-2) \to 45 \mathcal O_{\p^6} (-1) \to 15 \mathcal O_{\p^6} \to \mathcal E \to 0,
\]
respectively. The mapping cone gives a free resolution of $\mathcal I_{C/V_5}(3)$ as
\[
0 \to 2 \mathcal O_{\p^6}(-5) \to 25 \mathcal O_{\p^6}(-3) \to 55 \mathcal O_{\p^6} (-2) 
\]
\[
\to 45 \mathcal O_{\p^6}(-1)  \to 13 \mathcal O_{\p^6} \to \mathcal I_{C/V_5}(3) \to 0.
\]

We recover a free resolution of $C$ in $\p^6$ from the short exact sequence $0 \to \mathcal I_{V_5/\p^6} \to \mathcal I_{C/\p^6} \to \mathcal I_{C/V_5} \to 0$, and note that its Betti table coincides with the Betti table in the conclusion. Since $C \subset V_5$, the ideal of $C$ and $V_5$ share the same quadrics, and thus there are $5$ linear syzygies among these quadrics, hence, the free resolution of $C$ we obtained above is indeed minimal.

The last part of the conclusion follows immediately again by the fact that the ideal of $V_5$ in $\mathbb P^6$ is generated by five quadrics.
\end{proof}

Note that if we have such a curve $C \subset V_5$, then we can reverse the construction and obtain an Ulrich bundle of rank $3$ on $V_5$. To apply the Hartshorne-Serre correspondence, we only need to show that $\omega_C(-1)$ has two global sections which generate the graded module $H_{\ast}^0 (\omega_C)$. Since $C$ is an ACM curve, it suffices to establish the surjectivity of the multiplication maps
\[
H^0 (C, \omega_C (-1)) \otimes H^0 (C, \mathcal O_C (j)) \to H^0 (C, \mathcal \omega_C (j-1))
\]
for all $j > 0$. Thanks to the base--point--free pencil trick, it reduces to the surjectivity of the multiplication map 
\[
\mu : H^0 (C, \omega_C(-1)) \otimes H^0 (C, \mathcal O_C (1)) \to H^0 (C, \omega_C),
\]
which, by dimension reasons, is equivalent to the injectivity of $\mu$ in this case. By considering the dual resolution of $C$ obtained by applying $\Hom_{\p^6} ( -, \omega_{\p^6})$ which gives a minimal free resolution for $\omega_C$, we notice that the line bundle $\omega_C (-1)$ is globally generated by two sections. Furthermore, the kernel of the multiplication map $\mu$ equals to the Koszul cohomology group $K_{1,-1} (C, \omega_C, \mathcal O_C(1))$. By Green's duality \cite[Theorem 2.c.6]{Gre84}, \cite[p. 188]{Gre89}, we have an isomorphism
\[
K_{1,-1} (C, \omega_C, \mathcal O_C (1)) \simeq K_{4,3}(C, \mathcal O_C (1))^{\vee} = 0.
\]
Hence, such a curve $C$ provides a rank--$3$ bundle $\mathcal E$ on $V_5$ as an extension $0 \to \mathcal O_{V_5}^{\oplus 2} \to \mathcal E \to \mathcal I_{C/V_5} (3) \to 0$. It is straightforward that this $\mathcal E$ is Ulrich.

As pointed out in \cite[Section 5]{CFK23}, such a family of ACM curves $C \subset \p^6$ associated to rank $3$ Ulrich bundles on $V_5 \subset \p^6$ cannot dominate $\mathcal M_{14}$. This fact distinguishes the del Pezzo threefold $V_5$ from $V_3$ and $V_4$. 
Hence, it is natural to ask the following question.

\begin{ques}\label{ques:MainQuestion}
Describe geometric conditions on a curve $C$ of genus $14$ so that $C$ can be embedded into $C \hookrightarrow V_5 \subset \p^6$ and $C$ is the degeneracy locus of two sections of some Ulrich bundles of rank $3$ on $V_5$.
\end{ques}

Asnwering this question is the main objective of this paper.

\section{Vanishing resonance and embeddings of curves into $V_5$}
\label{sec:Res}

We show next that the answer to Question \ref{ques:MainQuestion} is related to the existence of a certain rank--2 vector bundle with trivial resonance. The first part of the section is algebraic in nature and is applicable to various other situations. We shall utilize these algebraic facts in the proof of the main result of the paper, Theorem \ref{thm:MainThmV5} that represents the content of the second part of this section.

\subsection{Resonance of a skew--symmetric matrix} We recast here the definition of resonance from \cite{PS15}. Instead of working with subspaces, for practical reasons,  we prefer working with maps and matrices.
We consider $V$ an $n$--dimensional complex vector space with $n\ge 4$, $W$ another complex vector space, and $\partial:\wedge^2 V^\vee \to W^\vee$ a linear map. 
Denote the kernel of $\partial$ by $K^\perp$. We define the \emph{resonance variety of $\partial$} to be the affine variety
\[
\mathcal{R}(\partial):=\{a \in V^{\vee} ~ \vert ~ \exists ~ b \in V^{\vee} \text{ such that } a \wedge b \neq 0, \text{ and } \partial(a \wedge b)=0\} \cup \{0\}.
\]

It coincides with the usual resonance variety (see \cite{PS15}) $\mathcal R(V,K)$ of the pair $(V,K)$ where $K^\vee$ is the image of the map $\partial$. It was shown in \cite{AFRS23} that the resonance carries a natural scheme structure, however, we will not be concerned with this finer structure here, and we simply consider the reduced structure.

\begin{exmp}
    Let $X$ be a complex algebraic variety, and let $\mathcal F$ be a vector bundle on $X$. Consider the second determinant map 
    \[
    d_2 : \wedge^2 H^0 (X, \mathcal F) \to H^0 (X, \wedge^2 \mathcal F).
    \]
    The resonance variety $\mathcal R (X, \mathcal F)$ of $\mathcal F$ is by definition $\mathcal R(d_2)$. In the rank--two case, the map $d_2$ coincides with the determinant map 
    \[
    \det:\wedge^2 H^0 (X, \mathcal F) \to H^0 (X, \det (\mathcal F)).
    \]
    We refer to \cite{AFRW24} and \cite{AFRS23} for an extensive discussion on the resonance of vector bundles and several applications.
\end{exmp}

We fix a basis $\{x_1,\ldots,x_n\}$ in $V$, denote by $\{e_1,\ldots,e_n\}$ its dual basis in $V^\vee$, and put $z_{ij}=e_i\wedge e_j$. Denote also by $A=(\ell_{ij})_{i,j=\overline{1,n}}$ the skew--symmetric matrix of linear forms on $W$ associated to the map $\partial$ in the chosen bases, i.e. $\ell_{ij}:=\partial(e_i \wedge e_j)\in W^\vee$. 

Recall from \cite{KS89} that $A$ is said to \emph{have a generalized zero} if there exists an invertible matrix $U$ such that the skew--symmetric matrix $UAU^t$ has a zero off the diagonal. It means, up to a base--change in $V$, the matrix associated to $\partial$ has a zero off the diagonal.  A $4 \times 4$ Pfaffian of such a matrix $UAU^{t}$ is called a \emph{generalized Pfaffian of $A$}. 

The generalized Pfaffians are directly connected to the resonance, as shown below.

\begin{thm}\label{thm:ResonancePfaffianLemma}
	Notation as above. The following are equivalent:
	\begin{itemize}
		\item[(i)] The resonance $\mathcal{R}(\partial)$ is nontrivial;
		\item[(ii)] The matrix $A$ has a generalized zero;
		\item[(iii)] There exists a $4\times 4$ generalized Pfaffian of $A$ of rank at most four.
	\end{itemize}
	\end{thm}

\proof
(i) $\Rightarrow$ (ii). Let $e_1$ and $e_2$ be two non--collinear vectors in $V^\vee$ such that $\partial(e_1\wedge e_2)=0$. We complete $\{e_1,e_2\}$ to a basis $\{e_1,\ldots e_n\}$ of $V^\vee$. In this new basis, the matrix associated to $\partial$ has a zero element as its $(1,2)$-th entry, and hence $A$ has a generalized zero.

\medskip

(ii) $\Rightarrow$ (iii).  Without loss of generality, we may assume $\ell_{12}=0$. Consider the submatrix 
\begin{equation}
	\label{eqn:4x4Paffian}
A'=\left(
\begin{array}{cccc}
	0			&	\ell_{12}	&	\ell_{13}	&	\ell_{14}\\
	-\ell_{12}	& 	0			&	\ell_{23}	&	\ell_{24}\\ 
	-\ell_{13}	&	-\ell_{23}	&	0			&	\ell_{34}\\
	-\ell_{14}	&	-\ell_{24}	&	-\ell_{34}	&	0       
\end{array}
\right)
\end{equation}
of $A$. Since $\ell_{12}=0$ it follows that its Pfaffian equals $\ell_{14}\ell_{23}-\ell_{13}\ell_{24}$ which is a quadric of rank $\le 4$.    

\medskip

(iii) $\Rightarrow$ (i). Recall that the definition of the resonance is independent on the choice of a basis in $V$, and hence it only depends on the equivalence class of $A$. In particular, we may assume that $A$ has a $4\times 4$ Pfaffian of rank $\le 4$ and we want to prove that $\mathbb{P}K^\perp\cap \Gr(2,V^\vee)\ne\emptyset$ in $\mathbb{P}(\bigwedge^2V^\vee)$.

Without loss of generality, we may assume $n=4$ and $A$ looks like in (\ref{eqn:4x4Paffian}), i.e. $A=A'$. In this case, the Grassmanian $\Gr(2,V)$ is the quadric in $\mathbb{P}(\bigwedge^2V)$ given by the equation $Q=0$ where $Q=z_{12}z_{34}-z_{13}z_{24}+z_{14}z_{23}$. Its dual quadric is the Grassmannian  $\Gr(2,V^\vee)$.

We note that the dimension of $K$ cannot be six, for in this case, the rank of the Pfaffian would be six, too. If the dimension of $K$ is $\le 4$, then the dimension of $K^\perp$ would be at least two, and hence $\mathbb{P}K^\perp\cap \Gr(2,V^\vee)\ne\emptyset$.

We may assume then that $K$ is five--dimensional, and hence $\mathbb{P}K^\perp$ is a point in $\mathbb{P}(\bigwedge^2V^\vee)$ which corresponds to a hyperplane $H$ of $\mathbb{P}(\bigwedge^2V)$. We prove that this point belongs to $\Gr(2,V^\vee)$, equivalently $H$ is tangent to the dual quadric $\Gr(2,V)$.

Via the identification of $K$ with the hyperplane $H$ via the dual of $\partial$, the rank of the Pfaffian of $A$ equals the rank of $Q|_H$ and the hypothesis implies that the hyperplane section $H\cap \Gr(2,V)$ is a singular quadric. This is only possible when $H$ is tangent to the quadric $\Gr(2,V)$, which we wanted to prove.
\endproof

\begin{rem}
In the previous result, the notion of generalized Pfaffians when $\dim V > 4$ is used in an essential way. We notice that the ranks of $4\times 4$--Pfaffians might change in the equivalence class of $A$. Let 
\[
A = \begin{pmatrix}
0 & x_0 & x_1 & x_2 & x_3 \\
-x_0 & 0 & x_2 & x_3 & x_4 \\
-x_1 & -x_2 & 0 & x_4 & x_5 \\
-x_2 & -x_3 & -x_4 & 0 & x_0+x_5 \\
-x_3 & -x_4 & -x_5 & -x_0 - x_5 & 0
\end{pmatrix}
\]
be a $5 \times 5$ skew-symmetric matrix of linear forms in $\p^5$. One can immediately check that the five $4 \times 4$-Pfaffians of $A$ are quadrics of rank $\ge 5$ (indeed, two are of rank $6$ and the other three are of rank $5$). 

On the other hand, these Pfaffian quadrics define a smooth del Pezzo surface $S_5 \subset \p^5$ of degree $5$. Since every smooth del Pezzo surface of degree $5$ are projectively equivalent to each other, and the ideal defining $S_5$ can be generated by $5$ Pfaffian quadrics of rank $\le 4$ after a suitable base-change \cites[Theorem 6.1]{MP23}[Proposition 6.3-(2)]{KMP23}. 
\end{rem}

\subsection{Curves on $V_5$}
We focus next on finding geometric conditions on $C$ so that $C$ can be embedded into a del Pezzo threefold $V_5 \subset \p^6$. Let us recall the generic behavior of a curve of genus $14$ and its projective model in $\p^6$ embedded by a complete linear system $|\mathcal O_C(1)|$ of degree $18$. Since $\mathcal O_C(j)$ is nonspecial for $j \ge 2$, we may read off its Hilbert function so that the Hilbert series of $C$ is
\[
HS( \mathcal O_C) = \frac{1-5t^2-8t^3 + 45t^4 - 56 t^5 + 25t^6 - 2t^8}{(1-t)^7}.
\]
Hence, we expect that such a generic curve $C \subset \p^6$ has the Betti table
\begin{equation}
  \label{eqn:BettiGeneric}  
\begin{array}{c|cccccc}
j \setminus i & 0&1&2&3&4&5 \\
\hline
0 & 1&-&-&-&-&- \\
1 & -&5&-&-&-&- \\
2 & -&8&45&56&25&- \\
3 & -&-&-&-&-&2
\end{array}
\end{equation}

These $5$ quadrics define a reducible curve of degree $32$ containing $C$ which is a union of $C$ and a residual curve $P$ of genus $8$ and degree $14$. Indeed, a general curve $C \subset \p^6$ of genus $14$ and degree $18$ has the same Betti table as above, and is obtained as the residual curve of a paracanonical curve $P \subset \p^6$ of genus $8$ and degree $14$ \cite[Theorem 4.5, 5.14]{Ver05}.

The key issue is to develop a criterion to distinguish whether or not a given curve $C \subset \p^6$ can be embedded into a del Pezzo threefold $V_5 \subset \p^6$. If this is the case, the $5$ quadrics defining $C$ will also define $V_5$, which is a linear section of the Grassmannian $\Gr(2,5)$. In particular, these quadrics can be obtained as Pfaffians of a $5 \times 5$ skew--symmetric matrix by the Buchsbaum-Eisenbud structure theorem \cite[Theorem 2.1]{BE77}. Hence, the core of the question is knowing how to extract information about Pfaffian quadrics that can generate (linear sections of) the Grassmannian $\Gr(2,5)$ from a given curve $C \subset \p^6$. We see that the resonance of a certain vector bundle on $C$ nicely captures the required information as in the next few paragraphs.

\begin{thm}\label{thm:MainThmV5}
Let $C \subset \p^6$ be an ACM curve of genus $14$, degree $18$, and $K_{4,3} (C, \mathcal O_C(1)) = 0$. Then the following are equivalent:
\begin{enumerate}[(i)]
\item $C \subset V_5$ for some del Pezzo threefold $V_5 \subset \p^6$;
\item $C \subset V_5$ for some del Pezzo threefold $V_5$, and is the zero locus of $2$ global sections of an Ulrich bundle $\mathcal E$ of rank $3$ on $V_5$;
\item there is a rank $2$ vector bundle $\mathcal F$ on $C$ such that $\wedge^2 \mathcal F = \mathcal O_C (1)$, $h^0 (C, \mathcal F) = 5$, and $\mathcal R(C, \mathcal F) = 0$.
\end{enumerate}
\end{thm}

\begin{proof}
(i) $\Rightarrow$ (ii). Since $C$ is ACM and $4$-regular, we must have 
\[
\dim K_{5,3} (C, \mathcal O_C(1))=2
\]
from the Hilbert function of $C \subset \p^6$. This implies that $\omega_C (-1)$ is globally generated and $h^0 (C, \omega_C (-1)) = 2$ by considering the dual resolution as we discussed above. Moreover, we have $K_{1,-1} (C, \omega_C, \mathcal O_C(1)) \simeq K_{4,3}(C, \mathcal O_C(1))^{\vee} = 0$, and thus the multiplication map
\[
\mu : H^0 (C, \omega_C (-1)) \otimes H^0 (\p^6, \mathcal O_{\p^6}(1)) \to H^0 (C, \omega_C)
\]
is an isomorphism. Hence, $C$ provides a vector bundle $\mathcal E$ of rank $3$ on $V_5$ so that
\[
0 \to \mathcal O_{V_5}^{\oplus 2} \to \mathcal E \to \mathcal I_{C/V_5} (3) \to 0
\]
by the Hartshorne-Serre correspondence. It is clear that $\mathcal E$ is Ulrich.

(ii) $\Rightarrow$ (iii). Recall that $V_5$ is a linear section of the Grassmannian $\Gr(2,5)$, so let us denote the (restriction of) universal subbundle on $V_5$ by $\mathcal U$. We choose $\mathcal F = \mathcal U^{\vee} \vert_C$ as the restriction of the dual of the universal bundle $\mathcal U$. Since $\wedge^2 \mathcal U^{\vee} = \mathcal O_{V_5} (1)$, the restriction $\mathcal F$ is a rank $2$ vector bundle with $\wedge^2 \mathcal F = \mathcal O_C (1)$ as well.

We first claim that $h^0 (V_5, \mathcal U^{\vee}) = h^0 (C, \mathcal F) = 5$. From the exact sequence
\[
0 \to \mathcal O_{V_5}^{\oplus 2} \to \mathcal E \to \mathcal I_{C/V_5} (3) \to 0,
\]
we have a $4$-term exact sequence
\[
0 \to \mathcal U^{\vee} (-3)^{\oplus 2} \to \mathcal E \otimes \mathcal U^{\vee} (-3) \to \mathcal U^{\vee} \to \mathcal F \to 0
\]
by tensoring $\mathcal U^{\vee} (-3)$. Since $H^1 (V_5, \mathcal U^{\vee}) = 0$, it is enough to show that $H^1 (V_5, \mathcal E \otimes \mathcal U^{\vee} (-3)) = H^2 (V_5, \mathcal U^{\vee} (-3)) = 0$ to verify the claim $h^0 (C, \mathcal F) = h^0 (V_5, \mathcal U^{\vee}) = 5$. It is well-known that the universal bundle $\mathcal U$ is an ACM bundle on $V_5$, and thus $\mathcal U^{\vee} \simeq \mathcal U (1)$ has no intermediate cohomology, in particular, $H^2 (V_5, \mathcal U^{\vee}(-3)) = 0$. For the vanishing of $H^1 (V_5, \mathcal E \otimes \mathcal U^{\vee} (-3))$, we use the correspondence between Ulrich bundles on $V_5$ and quiver representations in \cite[Proposition 3.4]{LP21}. In particular, we have the short exact sequence $0 \to \mathcal U^{\oplus 3} \to {\mathcal Q^{\vee}}^{\oplus 3} \to \mathcal E(-1) \to 0$, so the desired cohomology vanishing is provided when
\[
H^1 (V_5, \mathcal Q^{\vee} \otimes \mathcal U^{\vee} (-2)) = H^2 (V_5, \mathcal U \otimes \mathcal U^{\vee} (-2)) = 0.
\]
By the Serre duality, it is equivalent to the vanishing $H^2 (V_5, \mathcal Q \otimes \mathcal U) = H^1 (V_5, \mathcal U \otimes \mathcal U^{\vee}) = 0$. For the first vanishing, we have $H^2 (V_5, \mathcal Q \otimes \mathcal U) = \ext_{V_5}^2 (\mathcal Q^{\vee}, \mathcal U) = 0$ which is used to prove that $\mathcal U$ is left-orthogonal to $\mathcal Q^{\vee}$. The second vanishing is also the key for the fact that $\mathcal U$ is an exceptional bundle in $D^b (V_5)$, and thus $\ext_{V_5}^i (\mathcal U, \mathcal U) = 0$ for $i \neq 0$.

Next, we compute the resonance variety $\mathcal R (C, \mathcal F)$. Suppose that it is nonzero. Let $\det : \wedge^2 H^0 (C, \mathcal F) \to H^0 (C, \mathcal O_C(1)) \simeq H^0 (\p^6, \mathcal O_{\p^6} (1))$ be the determinant map. Note that $\mathcal F$ defines a nonzero Koszul cycle in $K_{2,1} (C, \mathcal O_C (1))$ as in \cites[Section 1]{Voi05}[Section 3.4.1]{AN10}. In particular, the $5$ quadrics in the ideal $I(C)$ of $C$ can be obtained as $4$-Pfaffians of the map $\det$. By Theorem \ref{thm:ResonancePfaffianLemma}, there is a $4$-Pfaffian of $\det$ of rank $\le 4$ among them. However, the quadrics in the ideal $I(C)$ and $I(V_5)$ are the same, and there is no quadric of rank $\le 4$ in $I(V_5)$ \cite[Proposition 6.3-(3)]{KMP23} which contradicts to the assumption. We conclude that the resonance $\mathcal R (C, \mathcal F) = 0$ is zero.

Note that such a vector bundle $\mathcal F$ does not exist over a general curve $C \subset \p^6$ of genus $14$ and degree $18$. Suppose that $C \subset \p^6$ has such a rank $2$ vector bundle $\mathcal F$. Then $\mathcal F$ defines a nonzero Koszul cycle in $K_{2,1} (C, \mathcal O_C(1))$, in particular, the $5$ quadrics in the ideal $I(C)$ of $C$ must allow a linear syzygy. On the other hand, the quadrics in the ideal of a general model $\widetilde{C} \subset \p^6$ of genus $14$ and degree $18$ have no linear syzygies among them. 

(iii) $\Rightarrow$ (i). Let $\det : \wedge^2 H^0 (C, \mathcal F) \to H^0 (C, \mathcal O_C (1)) \simeq H^0 (\p^6, \mathcal O_{\p^6} (1))$ be the determinant map, which is represented by a $5 \times 5$ skew-symmetric matrix composed of linear forms in $\p^6$. 
We follow the construction of Grassmannian syzygies introduced by Koh-Stillman \cite[Lemma 1.3]{KS89}. The five principal $4$-Pfaffians of $\det$ are contained in $I(C)$, we consider the variety $X$ defined by these Pfaffian quadrics which contains $C$ as its subvariety. It is clear that $X$ is a $\p^6$-linear section of the Grassmannian $\Gr(2,5)$, and hence it remains to show that $X$ is isomorphic to a smooth del Pezzo threefold $V_5 \subset \p^6$.

Since the resonance $\mathcal R (C, \mathcal F)$ is trivial, we see that the map $\det$ is non-degenerate and surjective by Theorem \ref{thm:ResonancePfaffianLemma}. Via this surjection, we have an inclusion $H^0 (\p^6, \mathcal O_{\p^6}(1))^{\vee} \subset \wedge^2 H^0 (C, \mathcal F)^{\vee}$. Moreover, $\mathcal R(C, \mathcal F)=0$ also implies
\[
Y = \mathbb{P} ( {H^0 (\mathcal O_{\p^6}(1))^{\vee}}^\perp ) ~ \cap~  \Gr(2, H^0 (\mathcal F)) 
\]
\[
= \p (\ker (\det)) ~ \cap ~ \Gr(2, H^0 (\mathcal F)) = \emptyset 
\]
in $\mathbb{P} (\wedge^2 H^0 (C, \mathcal F))$. 

By the projective duality, the linear section of the Grassmannian $X = \Gr(2, H^0 (C, \mathcal F)^{\vee}) \cap \mathbb{P} (H^0 (\p^6, \mathcal O_{\p^6}(1))^{\vee})$ is smooth and (dimensionally) transverse if and only if $Y$ is smooth and transverse \cite[Proposition 2.24]{DK18}. In our case, $H^0 (\p^6, \mathcal O_{\p^6}(1))^{\vee} \subset \wedge^2 H^0 (C, \mathcal F)^{\vee}$ is a linear subspace of codimension $3$, being $Y$ smooth and transverse is equivalent to $Y = \emptyset$ as similar as in \cite[Lemma 3.3]{Kuz18}. 
We conclude that $X$ is a smooth $3$-dimensional linear section of the Grassmannian $\Gr(2,5)$, and thus projectively equivalent to a smooth del Pezzo threefold $V_5 \subset \p^6$ \cite[Theorem II-1.1]{Isk80}. 
\end{proof}

\begin{rem} 
The condition $K_{4,3} (C, \mathcal O_C(1))=0$ is essential. Indeed, by following the instructions in \cite[Example 4.4]{AC00}, one can construct an ACM curve $C$ of genus $14$ and degree $18$ contained in a smooth del Pezzo threefold $V_5$ but $K_{4,3} (C, \mathcal O_C (1)) \neq 0$. In this example, two global sections of $\omega_C (-1)$ have a linear syzygy which implies $\dim K_{4,3} (C, \mathcal O_C (1)) = 1$, and such $C$ does not satisfy the generating condition to apply the Hartshorne-Serre correspondence. We refer to \cites[Remark 5.5]{CH12}[p. 17284]{CKL21}~ for details since the reasons are exactly the same.
\end{rem}

\begin{rem} By using similar techniques as in the proof above, one can also show that $h^0 (V_5, \mathcal U \otimes \mathcal U^{\vee}) = h^0 (C, \mathcal F \otimes \mathcal F^{\vee}) = 1$ where $\mathcal F$ is a rank $2$ vector bundle on $C$ that appeared in statement of the theorem. In particular, $\mathcal F$ is a simple vector bundle on $C$.
\end{rem}

\def\cprime{$'$} \def\cprime{$'$} \def\cprime{$'$} \def\cprime{$'$}
  \def\cprime{$'$} \def\cprime{$'$} \def\dbar{\leavevmode\hbox to
  0pt{\hskip.2ex \accent"16\hss}d} \def\cprime{$'$} \def\cprime{$'$}
  \def\polhk#1{\setbox0=\hbox{#1}{\ooalign{\hidewidth
  \lower1.5ex\hbox{`}\hidewidth\crcr\unhbox0}}} \def\cprime{$'$}
  \def\cprime{$'$} \def\cprime{$'$} \def\cprime{$'$}
  \def\polhk#1{\setbox0=\hbox{#1}{\ooalign{\hidewidth
  \lower1.5ex\hbox{`}\hidewidth\crcr\unhbox0}}} \def\cdprime{$''$}
  \def\cprime{$'$} \def\cprime{$'$} \def\cprime{$'$} \def\cprime{$'$}
\providecommand{\bysame}{\leavevmode\hbox to3em{\hrulefill}\thinspace}
\providecommand{\MR}{\relax\ifhmode\unskip\space\fi MR }
\providecommand{\MRhref}[2]{%
  \href{http://www.ams.org/mathscinet-getitem?mr=#1}{#2}
}

\vskip1cm

\end{document}